\documentclass[12pt]{amsart}
\usepackage{amsmath,amssymb,latexsym,cancel,rotating}
\usepackage{graphicx,amssymb,mathrsfs,amsmath,color,fancyhdr,amsthm}
\usepackage[all]{xy}

\newtheorem{theorem}{Theorem}[section]

\newtheorem{corollary}[theorem]{Corollary}

\newtheorem{proposition}[theorem]{Proposition}

\newtheorem{remark}[theorem]{Remark}

\newtheorem{thm}{Theorem}

 \def\cD{{\mathcal D}}

\def\bbN{{\mathbb N}}    \def\bbQ{{\mathbb Q}}

  \def\leq{\leqslant}  \def\geq{\geqslant}

\def\dim{\mbox{\rm dim}\,}

\def\bfV{{\mathbf V}}

\begin{document}

\title[Derivation functors]
{Derivation functors and Lusztig's induction functors}
\thanks{This work was supported by National Natural Science Foundation of China (No. 11771445)}

\author[Zhao]{Minghui Zhao}
\address{School of Science, Beijing Forestry University, Beijing 100083, P. R. China}
\email{zhaomh@bjfu.edu.cn (M.Zhao)}


\date{\today}


\bibliographystyle{abbrv}

\maketitle

\begin{abstract}
Lusztig proved the compatibility of induction functors and restriction functors for Lusztig's perverse sheaves. Fang-Lan-Xiao established a categorification of Green's formula and gave a sheaf-level proof of this compatibility for all semisimple complexes. As an application, we study the relation between induction functors and derivation functors, which is a kind of special restriction functors.
\end{abstract}

\section{Introduction}

\subsection{}

Let $\mathbf{U}$ be a quantum group associated to a quiver $Q=(I,H)$ and $\mathbf{U}^-$ be its negative part. The algebra $\mathbf{U}^-$ is defined by the generators $F_i$ for $i\in I$ and quantum Serre relations.
In \cite{Ringel_Hall_algebras_and_quantum_groups}, Ringel introduced the twisted Ringel-Hall algebra $H^{\ast}_q(Q)$ and proved that the composition subalgebra of $H^{\ast}_q(Q)$ is isomorphic to a specilization of $\mathbf{U}^-$. In \cite{green1995hall}, Green introduced the comultiplication and showed that the comultiplication is  an algebra homomorphism by using Green's formula.

Fix an $I$-graded vector space $\bfV=\bigoplus_{i\in I}{\mathbf{V}}_i$ with dimension vector $\nu\in\mathbb{N}I$ and denote by $E_{\bfV}$ the variety consisting of representations of the quiver $Q$.
There is a $G_{\mathbf{V}}=\prod_{i\in I}GL({\mathbf{V}}_i)$ action on $E_{\mathbf{V}}$.
Denote by $\cD_{G_\bfV}^b(E_\bfV)$ the $G_\bfV$-equivariant bounded derived category of $\overline{\bbQ}_l$-constructible mixed complexes on $E_\bfV$
and by $\cD_{G_\bfV}^{b,ss}(E_\bfV)$ the subcategory of $\cD_{G_\bfV}^b(E_\bfV)$ consisting of semisimple complexes (see \cite{Kiehl_Weissauer_Weil_conjectures_perverse_sheaves_and_l'adic_Fourier_transform,Bernstein_Lunts_Equivariant_sheaves_and_functors}).

Inspired by the work of Ringel,
Lusztig (\cite{Lusztig_Canonical_bases_arising_from_quantized_enveloping_algebra,Lusztig_Quivers_perverse_sheaves_and_the_quantized_enveloping_algebras}) gave a categorification of $\mathbf{U}_{\nu}^-$ by using an additive subcategory $\mathcal{Q}_{\nu}$ of $\cD_{G_\bfV}^{b,ss}(E_\bfV)$.
The multiplication and comultiplication in $\mathbf{U}^-$ are induced by induction functors $\mathrm{Ind}^{\alpha+\beta}_{\alpha,\beta}$
and  restriction functors $\mathrm{Res}^{\alpha+\beta}_{\alpha,\beta}$, respectively.

Let $K_\nu$ be the Grothendieck group of the category $\cD_{G_{\mathbf{V}}}^{b,ss}(E_{\mathbf{V}})$ and $\mathbf{K}=\bigoplus_{\nu\in\bbN I}K_\nu.$
The functors $A\ast B=\mathrm{Ind}^{\alpha+\beta}_{\alpha,\beta}(A\boxtimes B)$ induce a multiplication  on $\mathbf{K}$.
The algebra $\mathbf{K}$ contains $\mathbf{U}^-$ as a subalgebra (\cite{XXZ_Ringel-Hall_algebras_beyond_their_quantum_groups}).

Fix $\alpha,\beta,\alpha',\beta'\in\mathbb{N}I$ such that $\alpha+\beta=\alpha'+\beta'$.
Let $\mathcal{N}$ be the set of $(\alpha_1,\alpha_2,\beta_1,\beta_2)\in(\mathbb{N}I)^4$ such that $\alpha_1+\alpha_2=\alpha, \beta_1+\beta_2=\beta, \alpha_1+\beta_1=\alpha', \alpha_2+\beta_2=\beta'$.
For any $A\in\mathcal{Q}_{{\alpha}}$ and $B\in\mathcal{Q}_{{\beta}}$, Lusztig (\cite{Lusztig_Quivers_perverse_sheaves_and_the_quantized_enveloping_algebras}) proved the compatibility of induction functors and restriction functors:
\begin{eqnarray}\label{formula_1}
&&\mathrm{Res}^{\alpha'+\beta'}_{\alpha',\beta'}\mathrm{Ind}^{\alpha+\beta}_{\alpha,\beta}(A\boxtimes B)\\
&\simeq&\bigoplus_{\lambda\in\mathcal{N}}(\mathrm{Ind}^{\alpha'}_{\alpha_1,\beta_1}\times\mathrm{Ind}^{\beta'}_{\alpha_2,\beta_2})(\tau_{\lambda})_{!}(\mathrm{Res}^{\alpha}_{\alpha_1,\alpha_2}(A)\times\mathrm{Res}^{\beta}_{\beta_1,\beta_2}(B)) [-(\alpha_2,\beta_1)](-\frac{(\alpha_2,\beta_1)}{2})\nonumber.
\end{eqnarray}
Formula (\ref{formula_1}) implies that the comultiplication on $\mathbf{U}^-$ is a homomorphism of algebras.

By Green's formula and the sheaf-function correspondence, Xiao-Xu-Zhao (\cite{XXZ_Ringel-Hall_algebras_beyond_their_quantum_groups}) proved that Formula (\ref{formula_1}) holds for all $A\in\cD_{G_{\bfV_\alpha}}^{b,ss}(E_{\bfV_\alpha})$ and $B\in\cD_{G_{\bfV_\beta}}^{b,ss}(E_{\bfV_\beta})$.
In \cite{Fang_Lan_Xiao_The_parity_of_Lusztig_restriction_functor_and_Green_formula}, Fang-Lan-Xiao established a categorification of Green's formula and gave a sheaf-level proof of Formula (\ref{formula_1}) for all $A\in\cD_{G_{\bfV_\alpha}}^{b,ss}(E_{\bfV_\alpha})$ and $B\in\cD_{G_{\bfV_\beta}}^{b,ss}(E_{\bfV_\beta})$.

\subsection{}
For any $i\in I$, Lusztig (\cite{Lusztig_Introduction_to_quantum_groups}) introduced derivations ${_ir}:\mathbf{U}^-\rightarrow\mathbf{U}^-$ and ${r_i}:\mathbf{U}^-\rightarrow\mathbf{U}^-$ such that
\begin{equation}\label{formula_2}
{_ir}^{m}(xy)=\sum_{t=0}^{m}v^{(\nu-ti,(m-t)i)+t(m-t)}\frac{[m]_{v}!}{[t]_{v}![m-t]_{v}!}{_ir}^{t}(x){_ir}^{m-t}(y)
\end{equation}
and
\begin{equation}\label{formula_3}
{r^m_i}(xy)=\sum_{t=0}^{m}v^{(ti,\nu'-(m-t)i)+t(m-t)}\frac{[m]_{v}!}{[t]_{v}![m-t]_{v}!}{r_i^t}(x){r_i^{m-t}}(y)
\end{equation}
for $x\in\mathbf{U}_{\nu}^-$ and $y\in\mathbf{U}_{\nu'}^-$. As a generalization, Chen-Xiao (\cite{chen_xiao_1999}) introduced derivations ${_i\delta}:H^{\ast}_q(Q)\rightarrow H^{\ast}_q(Q)$ and ${\delta_i}:H^{\ast}_q(Q)\rightarrow H^{\ast}_q(Q)$ satisfying similar relations.

Denote by ${_{mi}\mathcal{R}}=\mathrm{Res}^{\alpha}_{mi,\alpha-mi}$ and ${\mathcal{R}_{mi}}=\mathrm{Res}^{\alpha}_{\alpha-mi,mi}$ the derivation functors, which induce derivations ${_ir}$ and ${r_i}$ on $\mathbf{K}$ such that $${_ir^{m}}([A])=v^{\frac{m(m-1)}{2}}[D({_{mi}\mathcal{R}}(DA))]$$
and $${r_i^{m}}([A])=v^{\frac{m(m-1)}{2}}[D({\mathcal{R}_{mi}}(DA))]$$ for any $[A]\in\mathbf{K}$, where $D$ is the Verdier duality on $\cD_{G_{\bfV}}^{b}(E_{\bfV})$.

Fix $A\in\cD_{G_{\bfV_\alpha}}^{b,ss}(E_{\bfV_\alpha})$ and $B\in\cD_{G_{\bfV_\beta}}^{b,ss}(E_{\bfV_\beta})$.
Formula (\ref{formula_1}) implies that
\begin{equation}\label{formula_4}{_{mi}\mathcal{R}}(A\ast B)\simeq\bigoplus_{t=a}^{b}\mathbf{1}_{pt}^{f_{m,t}(v)}\boxtimes({_{ti}\mathcal{R}}(A)\ast{_{(m-t)i}\mathcal{R}}(B))[-P_t](-\frac{P_t}{2}),\end{equation}
and
\begin{equation}\label{formula_5}{\mathcal{R}_{mi}}(A\ast B)\simeq\bigoplus_{t=a}^{b}\mathbf{1}_{pt}^{f_{m,t}(v)}\boxtimes({\mathcal{R}_{ti}}(A)\ast{\mathcal{R}_{(m-t)i}}(B))[-P'_t](-\frac{P'_t}{2}),\end{equation}
when $\alpha'=mi$ and $\beta'=mi$, respectively.
Here, $\alpha=\sum_{i\in I}\alpha_ii$, $\beta=\sum_{i\in I}\beta_ii$, $a=\max\{0,m-\beta_i\}$, $b=\min\{m,\alpha_i\}$, $P_t=(\alpha-ti,(m-t)i)$, $P'_t=(ti,\beta-(m-t)i)$ and $f_{m,t}(v)=\frac{[m]_{v}!}{[t]_{v}![m-t]_{v}!}$.

Formulas (\ref{formula_4}) and (\ref{formula_5}) imply that derivations ${_ir}$ and ${r_i}$ on $\mathbf{K}$ satisfy  Formulas (\ref{formula_2}) and (\ref{formula_3}).

\subsection{}

The main result of this paper is the following theorems.

\begin{thm}\label{MT_intro_1}
For any $A\in\cD_{G_{\bfV_\alpha}}^{b}(E_{\bfV_\alpha})$ and $B\in\cD_{G_{\bfV_\beta}}^{b}(E_{\bfV_\beta})$,
we have distinguished triangles
$${_{mi}\mathcal{R}}(A\ast B)_{\lambda_t}\rightarrow{_{mi}\mathcal{R}}(A\ast B)_{\geq t}\rightarrow {_{mi}\mathcal{R}}(A\ast B)_{\geq t+1}\rightarrow,$$
in $\cD_{G_{\bfV_{\beta'}}}^{b}(E_{\bfV_{\beta'}})$ for $t=a,a+1,\ldots,b-1$, such that
\begin{enumerate}
  \item[(1)]${_{mi}\mathcal{R}}(A\ast B)_{\geq a}={_{mi}\mathcal{R}}(A\ast B)$;
  \item[(2)]${_{mi}\mathcal{R}}(A\ast B)_{\lambda_t}\simeq\mathbf{1}_{pt}^{f_{m,t}(v)}\boxtimes({_{ti}\mathcal{R}}(A)\ast {_{(m-t)i}\mathcal{R}}(B))[-P_{t}](\frac{-P_{t}}{2})$ for $t=a,a+1,\ldots,b$;
  \item[(3)]${_{mi}\mathcal{R}}(A\ast B)_{\geq b}={_{mi}\mathcal{R}}(A\ast B)_{\lambda_b}$.
\end{enumerate}
\end{thm}

\begin{thm}\label{MT_intro_1_1}
For any $A\in\cD_{G_{\bfV_\alpha}}^{b}(E_{\bfV_\alpha})$ and $B\in\cD_{G_{\bfV_\beta}}^{b}(E_{\bfV_\beta})$,
we have distinguished triangles
$${\mathcal{R}_{mi}}(A\ast B)_{\lambda_t}\rightarrow{\mathcal{R}_{mi}}(A\ast B)_{\leq t}\rightarrow {\mathcal{R}_{mi}}(A\ast B)_{\leq t-1}\rightarrow,$$
in $\cD_{G_{\bfV_{\alpha'}}}^{b}(E_{\bfV_{\alpha'}})$ for $t=a+1,\ldots,b-1,b$, such that
\begin{enumerate}
  \item[(1)]${\mathcal{R}_{mi}}(A\ast B)_{\leq b}={\mathcal{R}_{mi}}(A\ast B)$;
  \item[(2)]${\mathcal{R}_{mi}}(A\ast B)_{\lambda_t}\simeq\mathbf{1}_{pt}^{f_{m,t}(v)}\boxtimes({\mathcal{R}_{ti}}(A)\ast {\mathcal{R}_{(m-t)i}}(B))[-P'_{t}](\frac{-P'_{t}}{2})$ for $t=a,a+1,\ldots,b$;
  \item[(3)]${\mathcal{R}_{mi}}(A\ast B)_{\leq a}={\mathcal{R}_{mi}}(A\ast B)_{\lambda_a}$.
\end{enumerate}
\end{thm}

The proofs of Theorems \ref{MT_intro_1} and \ref{MT_intro_1_1} are based on the proof of Formula (\ref{formula_1}) of Fang-Lan-Xiao in \cite{Fang_Lan_Xiao_The_parity_of_Lusztig_restriction_functor_and_Green_formula}.

In this paper, we also show that derivation functors satisfy quantum Serre relations.
\begin{thm}\label{MT_intro_2}
For any $i\neq j\in I$, we have
$$\bigoplus_{m+n=1-(i,j)\atop\textrm{ $m$ is odd}}{_{mi}\mathcal{R}}\cdot{_{j}\mathcal{R}}\cdot{_{ni}\mathcal{R}}\simeq\bigoplus_{m+n=1-(i,j)\atop\textrm{ $m$ is even}}{_{mi}\mathcal{R}}\cdot{_{j}\mathcal{R}}\cdot{_{ni}\mathcal{R}}$$ and
$$\bigoplus_{m+n=1-(i,j)\atop\textrm{ $m$ is odd}}{\mathcal{R}_{mi}}\cdot{\mathcal{R}_{j}}\cdot{\mathcal{R}_{ni}}\simeq\bigoplus_{m+n=1-(i,j)\atop\textrm{ $m$ is even}}{\mathcal{R}_{mi}}\cdot{\mathcal{R}_{j}}\cdot{\mathcal{R}_{ni}}.$$
\end{thm}


In Section 2, we shall recall the definitions of restriction functors and induction functors. In Section 3, we shall recall the definition of derivation functors. The main results and their proofs are given in Section 4.

\section{Restriction functors and induction functors}

\subsection{}

Let $Q=(I,H,s,t)$ be an acyclic quiver, where $I$ is the set of vertices, $H$ is the set of arrows, and $s,t:H\rightarrow I$ are two maps such that $s(h)$ (resp. $t(h)$) is the source (resp. target) of $h\in H$. Denote by $\mathbb{F}_q$ a
finite field with $q$ elements and let $\mathbb{K}=\overline{\mathbb{F}}_q$. For any $I$-graded vector space $\mathbf{V}=\bigoplus_{i\in I}\mathbf{V}_i$ over $\mathbb{K}$ with dimension vector $\nu=\sum_{i\in I}(\dim_{\mathbb{K}}\mathbf{V}_i)i\in\mathbb{N}I$,
let $$E_{\mathbf{V}}=\bigoplus_{\rho\in H}\mathrm{Hom}_\mathbb{K}(\bfV_{s(\rho)},\bfV_{t(\rho)})$$
and $$G_{\mathbf{V}}=\prod_{i\in I}GL_\mathbb{K}({\mathbf{V}}_i).$$
The group $G_{\mathbf{V}}$ acts on $E_{\mathbf{V}}$ by $g.x=(g_{t(h)}x_{h}g^{-1}_{s(h)})_{h\in H}$ for any $g=(g_i)_{i\in I}$ and $x=(x_h)_{h\in H}$.

By $\cD_{G_\bfV}^b(E_\bfV)$, we denote the $G_\bfV$-equivariant bounded derived category of $\overline{\bbQ}_l$-constructible mixed complexes on $E_\bfV$ and by $\cD_{G_\bfV}^{b,ss}(E_\bfV)$ the subcategory of $\cD_{G_\bfV}^b(E_\bfV)$ consisting of mixed semisimple complexes.
Let $K_\nu$ be the Grothendieck group of the category $\cD_{G_{\mathbf{V}}}^{b,ss}(E_{\mathbf{V}})$,
which is a $\mathbb{Z}[v,v^{-1}]$-module by $v^{\pm}[L]=[L[\pm1](\pm\frac{1}{2})]$.
Let $$\mathbf{K}=\bigoplus_{\nu\in\bbN I}K_\nu.$$
For $f(v)=\sum_{l=m}^{n}a_lv^l\in\mathbb{Z}_{\geq0}[v,v^{-1}]$ and $L\in\cD_{G_{\mathbf{V}}}^{b}(E_{\mathbf{V}})$, denote
$$L^{f(v)}=\bigoplus_{l=m}^{n}L^{\oplus a_l}[l](\frac{l}{2}).$$

\subsection{}

For any $\alpha,\beta\in\bbN I$, fix $I$-graded vector spaces $\bfV_{\alpha+\beta}$, $\bfV_{\alpha}$, $\bfV_{\beta}$ with dimension vectors $\alpha+\beta,\alpha,\beta$, respectively.

Consider the following diagram (Section 3 in \cite{Lusztig_Quivers_perverse_sheaves_and_the_quantized_enveloping_algebras})
\begin{equation*}
\xymatrix{E_{\bfV_{\alpha}}\times E_{\bfV_{\beta}}&E_{\alpha,\beta}'\ar[l]_-{p_{1\alpha,\beta}}\ar[r]^-{p_{2\alpha,\beta}}&E_{\alpha,\beta}''\ar[r]^-{p_{3\alpha,\beta}}&E_{\bfV_{\alpha+\beta}}}.
\end{equation*}
Here
\begin{enumerate}
  \item[(1)]$E_{\alpha,\beta}''=\{(x,\mathbf{W})\}$, where $x\in E_{\bfV_{\alpha+\beta}}$ and $\mathbf{W}$ is an $x$-stable subspace of $\bfV_{\alpha+\beta}$ with dimension vector $\beta$;
  \item[(2)]$E_{\alpha,\beta}'=\{(x,\mathbf{W},R'',R')\}$, where $(x,\mathbf{W})\in E_{\alpha,\beta}''$, $R'':\bfV_{\beta}\simeq \mathbf{W}$ and $R':\bfV_{\alpha}\simeq{\bfV_{\alpha+\beta}}/\mathbf{W}$;
  \item[(3)]$p_{1\alpha,\beta}(x,\mathbf{W},R'',R')=(x',x'')$, where $x'\in E_{\bfV_{\alpha}}$ such that $R'_{t(h)}x'_{h}=\bar{x}_{h}R'_{s(h)}$ ($\bar{x}_{h}:(\bfV_{\alpha+\beta})_{s(h)}/\mathbf{W}_{s(h)}\rightarrow(\bfV_{\alpha+\beta})_{t(h)}/\mathbf{W}_{t(h)}$ is induced by $x_h$) and $x''\in E_{\bfV_{\beta}}$ such that $R''_{t(h)}x''_{h}=x_{h}|_{\mathbf{W}_{s(h)}}R''_{s(h)}$;
  \item[(4)]$p_{2\alpha,\beta}(x,\mathbf{W},R'',R')=(x,\mathbf{W})$;
  \item[(5)]$p_{3\alpha,\beta}(x,\mathbf{W})=x$.
\end{enumerate}
Since $p_{2\alpha,\beta}$ is a $G_{\bfV_{\alpha}}\times G_{\bfV_{\beta}}$-principal bundle, $$p^\ast_{2\alpha,\beta}:\cD^b_{G_{\bfV_{\alpha+\beta}}}(E_{\alpha,\beta}'')\rightarrow\cD^b_{G_{\bfV_{\alpha+\beta}}\times G_{\bfV_{\alpha}}\times G_{\bfV_{\beta}}}(E_{\alpha,\beta}')$$ is an equivalence of categories.
The inverse of $p^\ast_{2\alpha,\beta}$ is denoted by $({p_{2\alpha,\beta}})_{\flat}$.

The induction functor is defined as
\begin{eqnarray*}
\mathrm{Ind}^{\alpha+\beta}_{\alpha,\beta}:\cD_{G_{\bfV_{\alpha}}\times G_{\bfV_{\beta}}}^b(E_{\bfV_{\alpha}}\times E_{\bfV_{\beta}})&\rightarrow&\cD_{G_{\bfV_{\alpha+\beta}}}^b(E_{\bfV_{\alpha+\beta}})\\
A&\mapsto&(p_{3\alpha,\beta})_{!}({p_{2\alpha,\beta}})_{\flat}p_{1\alpha,\beta}^{\ast}(A)[m_{\alpha,\beta}](\frac{m_{\alpha,\beta}}{2}),
\end{eqnarray*}
where $m_{\alpha,\beta}=\sum_{i\in I}\alpha_i\beta_i+\sum_{\rho\in H}\alpha_{s(\rho)}\beta_{t(\rho)}.$

The induction functor induces the following functor
\begin{eqnarray*}
\ast:\cD_{G_{\bfV_{\alpha}}}^b(E_{\bfV_{\alpha}})\times\cD_{G_{\bfV_{\beta}}}^b(E_{\bfV_{\beta}})&\rightarrow&\cD_{G_{\bfV_{\alpha+\beta}}}^b(E_{\bfV_{\alpha+\beta}})\\
(A,B)&\mapsto&A\ast B=\mathrm{Ind}^{\alpha+\beta}_{\alpha,\beta}(A\boxtimes B).
\end{eqnarray*}

For any $A\in\cD_{G_{\bfV_{\alpha}}}^{b,ss}(E_{\bfV_{\alpha}})$ and $B\in\cD_{G_{\bfV_{\beta}}}^{b,ss}(E_{\bfV_{\beta}})$,
the complex  $A\ast B$ is semisimple (Section 2.4 in \cite{Fang_Lan_Xiao_The_parity_of_Lusztig_restriction_functor_and_Green_formula}).
\begin{proposition}[Section 4 in \cite{XXZ_Ringel-Hall_algebras_beyond_their_quantum_groups}]\label{prop_associative}
Induction functors for various $\alpha,\beta\in\mathbb{N}I$ induce an associative multiplication structure on $\mathbf{K}$ by $[A][B]=[A\ast B]$.
\end{proposition}

\begin{remark}
In \cite{Lusztig_Canonical_bases_arising_from_quantized_enveloping_algebra,Lusztig_Quivers_perverse_sheaves_and_the_quantized_enveloping_algebras}, Lusztig  gave a categorification of  the negative part $\mathbf{U}^-$ of a quantum group by using an additive subcategory $\mathcal{Q}_{\nu}$ of $\cD_{G_\bfV}^{b,ss}(E_\bfV)$.
The multiplication on $\mathbf{U}^-$ is induced by induction functors $\mathrm{Ind}^{\alpha+\beta}_{\alpha,\beta}$ and
the algebra $\mathbf{K}$ contains  $\mathbf{U}^-$  as a subalgebra.
\end{remark}

\subsection{}

For any $\alpha,\beta\in\bbN I$, fix $I$-graded vector spaces $\bfV_{\alpha+\beta}$, $\bfV_{\alpha}$, $\bfV_{\beta}$ with dimension vectors $\alpha+\beta,\alpha,\beta$, respectively.
Fix a subspace $\mathbf{W}\subset\bfV_{\alpha+\beta}$ with dimension vector $\beta$ and
a pair of isomorphisms $R'':\bfV_{\beta}\simeq \mathbf{W}$ and $R':\bfV_{\alpha}\simeq\bfV_{\alpha+\beta}/\mathbf{W}$.

Consider the following diagram (Section 4 in \cite{Lusztig_Quivers_perverse_sheaves_and_the_quantized_enveloping_algebras})
\begin{equation*}
\xymatrix{E_{\bfV_{\alpha}}\times E_{\bfV_{\beta}}&F_{\alpha,\beta}\ar[l]_-{\kappa_{\alpha,\beta}}\ar[r]^-{\iota_{\alpha,\beta}}&E_{\bfV_{\alpha+\beta}}}.
\end{equation*}
Here
\begin{enumerate}
  \item[(1)]$F_{\alpha,\beta}=\{x\in E_{\bfV_{\alpha+\beta}}|x_{h}(\mathbf{W}_{s(h)})\subset \mathbf{W}_{t(h)}\}$;
  \item[(2)]$\kappa_{\alpha,\beta}(x)=p_{1\alpha,\beta}(x,\mathbf{W},R'',R')$;
  \item[(3)]$\iota_{\alpha,\beta}$ is the natural embedding.
\end{enumerate}

The restriction functor is defined as
\begin{eqnarray*}
\mathrm{Res}^{\alpha+\beta}_{\alpha,\beta}:\cD_{G_{\bfV_{\alpha+\beta}}}^b(E_{\bfV_{\alpha+\beta}})&\rightarrow&\cD_{G_{\bfV_{\alpha}}\times G_{\bfV_{\beta}}}^b(E_{\bfV_{\alpha}}\times E_{\bfV_{\beta}})\\
A&\mapsto&(\kappa_{\alpha,\beta})_!\iota_{\alpha,\beta}^\ast (A)[-\langle\alpha,\beta\rangle](-\frac{\langle\alpha,\beta\rangle}{2}),
\end{eqnarray*}
where $\langle\alpha,\beta\rangle=\sum_{i\in I}\alpha_i\beta_i-\sum_{\rho\in H}\alpha_{s(\rho)}\beta_{t(\rho)}.$

For any $A\in\cD_{G_{\bfV_{\alpha+\beta}}}^{b,ss}(E_{\bfV_{\alpha+\beta}})$,
the complex  $\mathrm{Res}^{\alpha+\beta}_{\alpha,\beta}(A)$ is semisimple (Section 2.5 in \cite{Fang_Lan_Xiao_The_parity_of_Lusztig_restriction_functor_and_Green_formula}).

\begin{remark}
The comultiplication on $\mathbf{U}^-$ is induced by restriction functors $\mathrm{Res}^{\alpha+\beta}_{\alpha,\beta}$.
\end{remark}

\subsection{}

Fix $\alpha,\beta,\alpha',\beta'\in\mathbb{N}I$ such that $\alpha+\beta=\alpha'+\beta'$ and ${\mathbf{V}_{\alpha}},{\mathbf{V}_{\beta}},{\mathbf{V}_{\alpha'}},{\mathbf{V}_{\beta'}}$ with dimension vectors $\alpha,\beta,\alpha',\beta'$, respectively.
Let $$\mathcal{N}=\{\lambda=(\alpha_1,\alpha_2,\beta_1,\beta_2)\,\,|\,\,\alpha_1+\alpha_2=\alpha, \beta_1+\beta_2=\beta, \alpha_1+\beta_1=\alpha', \alpha_2+\beta_2=\beta'\}.$$ For any $\lambda\in\mathcal{N}$,
fix ${\mathbf{V}_{\alpha_1}},{\mathbf{V}_{\alpha_2}},{\mathbf{V}_{\beta_1}},{\mathbf{V}_{\beta_2}}$ with dimension vectors $\alpha_1,\alpha_2,\beta_1,\beta_2$, respectively. Let
$$\tau_{\lambda}:E_{\mathbf{V}_{\alpha_1}}\times E_{\mathbf{V}_{\alpha_2}}\times E_{\mathbf{V}_{\beta_1}}\times E_{\mathbf{V}_{\beta_2}}\rightarrow E_{\mathbf{V}_{\alpha_1}}\times E_{\mathbf{V}_{\beta_1}}\times E_{\mathbf{V}_{\alpha_2}}\times E_{\mathbf{V}_{\beta_2}}$$ be the isomorphism sending $(x_1,x_2,y_1,y_2)$ to $(x_1,y_1,x_2,y_2)$.

For any $\mu,\nu\in\mathbb{N}I$, let $(\mu,\nu)=\langle\mu,\nu\rangle+\langle\nu,\mu\rangle.$

\begin{theorem}[Theorem 7 in \cite{XXZ_Ringel-Hall_algebras_beyond_their_quantum_groups}, Theorem 3.1 in \cite{Fang_Lan_Xiao_The_parity_of_Lusztig_restriction_functor_and_Green_formula}]\label{FLX_thm}
For any $A\in\cD_{G_{\bfV_\alpha}}^{b,ss}(E_{\bfV_\alpha})$ and $B\in\cD_{G_{\bfV_\beta}}^{b,ss}(E_{\bfV_\beta})$, we have
\begin{eqnarray*}&&\mathrm{Res}^{\alpha'+\beta'}_{\alpha',\beta'}\mathrm{Ind}^{\alpha+\beta}_{\alpha,\beta}(A\boxtimes B)\\
&\simeq&\bigoplus_{\lambda\in\mathcal{N}}(\mathrm{Ind}^{\alpha'}_{\alpha_1,\beta_1}\times\mathrm{Ind}^{\beta'}_{\alpha_2,\beta_2})(\tau_{\lambda})_{!}(\mathrm{Res}^{\alpha}_{\alpha_1,\alpha_2}(A)\times\mathrm{Res}^{\beta}_{\beta_1,\beta_2}(B)) [-(\alpha_2,\beta_1)](-\frac{(\alpha_2,\beta_1)}{2}).
\end{eqnarray*}
\end{theorem}

\begin{remark}
In Section 8 of \cite{Lusztig_Quivers_perverse_sheaves_and_the_quantized_enveloping_algebras}, Lusztig proved Theorem \ref{FLX_thm} for any $A\in\mathcal{Q}_{{\alpha}}$ and $B\in\mathcal{Q}_{{\beta}}$, which implies that the comultiplication on $\mathbf{U}^-$ is a homomorphism of algebras.
\end{remark}

\section{Derivation functors}

\subsection{}

For any $\alpha\in\bbN I$, fix $I$-graded vector spaces $\bfV_{\alpha}$, $\bfV_{\alpha-mi}$ and $\bfV_{mi}$ with dimension vectors $\alpha,\alpha-mi$ and $mi$, respectively.
Fix a subspace $\mathbf{W}\subset\bfV_{\alpha}$ with dimension vector $\alpha-mi$ and
an isomorphism $R'':\bfV_{\alpha-mi}\simeq \mathbf{W}$.

Consider the following diagram
\begin{equation*}
\xymatrix{E_{\bfV_{\alpha-mi}}&F_{mi,\alpha-mi}\ar[l]_-{\kappa_{mi,\alpha-mi}}\ar[r]^-{\iota_{mi,\alpha-mi}}&E_{\bfV_{\alpha}}}.
\end{equation*}
Here $\kappa_{mi,\alpha-mi}(x)=x''$, where $x''\in E_{\bfV_{\alpha-mi}}$ such that $R''_{t(h)}x''_{h}=x_{h}|_{\mathbf{W}_{s(h)}}R''_{s(h)}$.

The derivation functor is defined as
\begin{eqnarray*}
{_{mi}\mathcal{R}}:\cD_{G_{\bfV_{\alpha}}}^b(E_{\bfV_{\alpha}})&\rightarrow&\cD_{G_{\bfV_{\alpha-mi}}}^b(E_{\bfV_{\alpha-mi}})\\
A&\mapsto&(\kappa_{mi,\alpha-mi})_!\iota_{mi,\alpha-mi}^\ast A[-\langle{mi,\alpha-mi}\rangle](-\frac{\langle{mi,\alpha-mi}\rangle}{2}).
\end{eqnarray*}


Note that $\mathrm{Res}^{\alpha}_{mi,\alpha-mi}(A)=L_{mi}\boxtimes{_{mi}\mathcal{R}}(A)$,
where $L_{mi}=\mathbf{1}_{E_{\bfV_{mi}}}$ is the constant sheaf on $E_{\bfV_{mi}}$.

For any $m\in\mathbb{Z}_{>0}$, denote $[m]_{v}=\frac{v^m-v^{-m}}{v-v^{-1}}$ and $[m]_{v}!=\prod_{l=1}^{m}[l]_{v}$. Let $[0]_{v}!=1$. Denote $f_{m,t}(v)=\frac{[m]_{v}!}{[t]_{v}![m-t]_{v}!}$.
In Theorem \ref{FLX_thm}, assume that $\alpha'=mi$ and we get the following corollary.

\begin{corollary}\label{FLX_cor}
For any $A\in\cD_{G_{\bfV_\alpha}}^{b,ss}(E_{\bfV_\alpha})$ and $B\in\cD_{G_{\bfV_\beta}}^{b,ss}(E_{\bfV_\beta})$, we have
$${_{mi}\mathcal{R}}(A\ast B)\simeq\bigoplus_{t=a}^{b}\mathbf{1}_{pt}^{f_{m,t}(v)}\boxtimes({_{ti}\mathcal{R}}(A)\ast{_{(m-t)i}\mathcal{R}}(B))[-P_t](-\frac{P_t}{2}),$$
where $\alpha=\sum_{i\in I}\alpha_ii$, $\beta=\sum_{i\in I}\beta_ii$, $a=\max\{0,m-\beta_i\}$, $b=\min\{m,\alpha_i\}$, $P_t=(\alpha-ti,(m-t)i)$.
\end{corollary}

\begin{proof}
Theorem \ref{FLX_thm} implies that
\begin{eqnarray*}
&&L_{mi}\boxtimes{_{mi}\mathcal{R}}(A\ast B)\\
&\simeq&\bigoplus_{t=a}^{b}(L_{ti}\ast L_{(m-t)i})\boxtimes({_{ti}\mathcal{R}}(A)\ast{_{(m-t)i}\mathcal{R}}(B))[-P_t](-\frac{P_t}{2})\\
&\simeq&\bigoplus_{t=a}^{b}L^{f_{m,t}(v)}_{mi}\boxtimes({_{ti}\mathcal{R}}(A)\ast{_{(m-t)i}\mathcal{R}}(B))[-P_t](-\frac{P_t}{2})\\
&\simeq&L_{mi}\boxtimes\bigoplus_{t=a}^{b}\mathbf{1}_{pt}^{f_{m,t}(v)}\boxtimes({_{ti}\mathcal{R}}(A)\ast{_{(m-t)i}\mathcal{R}}(B))[-P_t](-\frac{P_t}{2}).
\end{eqnarray*}

\end{proof}

Derivation functors for various $\alpha\in\mathbb{N}I$ induce a derivation ${_ir}$ on $\mathbf{K}$ such that $${_ir^{m}}([A])=v^{\frac{m(m-1)}{2}}[D({_{mi}\mathcal{R}}(DA))]$$ for any $[A]\in\mathbf{K}$.

\begin{remark}
For any $i\in I$, Lusztig (Section 1.2.13 in \cite{Lusztig_Introduction_to_quantum_groups}) introduced a derivation ${_ir}:\mathbf{U}^-\rightarrow\mathbf{U}^-$ such that
\begin{equation}\label{formula_derivation}
{_ir}^{m}(xy)=\sum_{t=0}^{m}v^{(\nu-ti,(m-t)i)+t(m-t)}\frac{[m]_{v}!}{[t]_{v}![m-t]_{v}!}{_ir}^{t}(x){_ir}^{m-t}(y)
\end{equation}
for $x\in\mathbf{U}_{\nu}^-$.
Note that ${_ir}:\mathbf{U}^-\rightarrow\mathbf{U}^-$ is the restriction of ${_ir}:\mathbf{K}\rightarrow\mathbf{K}$ on $\mathbf{U}^-$ and
Corollary \ref{FLX_cor} is a categorification of Formula (\ref{formula_derivation}).
\end{remark}

\subsection{}

For any $\alpha\in\bbN I$, fix $I$-graded vector spaces $\bfV_{\alpha}$, $\bfV_{\alpha-mi}$ and $\bfV_{mi}$ with dimension vectors $\alpha,\alpha-mi$ and $mi$, respectively.
Fix a subspace $\mathbf{W}\subset\bfV_{\alpha}$ with dimension vector $mi$ and
an isomorphisms $R':\bfV_{\alpha-mi}\simeq \bfV_{\alpha}/\mathbf{W}$.

Consider the following diagram
\begin{equation*}
\xymatrix{E_{\bfV_{\alpha-mi}}&F_{\alpha-mi,mi}\ar[l]_-{\kappa_{\alpha-mi,mi}}\ar[r]^-{\iota_{\alpha-mi,mi}}&E_{\bfV_{\alpha}}}.
\end{equation*}
Here $\kappa_{\alpha-mi,mi}(x)=x'$, where $x'\in E_{\bfV_{\alpha-mi}}$ such that $R'_{t(h)}x'_{h}=\bar{x}_{h}R'_{s(h)}$ ($\bar{x}_{h}:(\bfV_{\alpha})_{s(h)}/\mathbf{W}_{s(h)}\rightarrow(\bfV_{\alpha})_{t(h)}/\mathbf{W}_{t(h)}$ is induced by $x_h$).

The derivation functor is defined as
\begin{eqnarray*}
{\mathcal{R}_{mi}}:\cD_{G_{\bfV_{\alpha}}}^b(E_{\bfV_{\alpha}})&\rightarrow&\cD_{G_{\bfV_{\alpha-mi}}}^b(E_{\bfV_{\alpha-mi}})\\
A&\mapsto&(\kappa_{\alpha-mi,mi})_!\iota_{\alpha-mi,mi}^\ast A[-\langle{\alpha-mi,mi}\rangle](-\frac{\langle{\alpha-mi,mi}\rangle}{2}).
\end{eqnarray*}


Note that $\mathrm{Res}^{\alpha}_{\alpha-mi,mi}(A)={\mathcal{R}_{mi}}(A)\boxtimes L_{mi}$.
In Theorem \ref{FLX_thm}, assume that $\beta'=mi$ and we get the following corollary.

\begin{corollary}\label{FLX_cor_1}
For any $A\in\cD_{G_{\bfV_\alpha}}^{b,ss}(E_{\bfV_\alpha})$ and $B\in\cD_{G_{\bfV_\beta}}^{b,ss}(E_{\bfV_\beta})$, we have
$${\mathcal{R}_{mi}}(A\ast B)\simeq\bigoplus_{t=a}^{b}\mathbf{1}_{pt}^{f_{m,t}(v)}\boxtimes({\mathcal{R}_{ti}}(A)\ast{\mathcal{R}_{(m-t)i}}(B))[-P'_t](-\frac{P'_t}{2}),$$
where $\alpha=\sum_{i\in I}\alpha_ii$, $\beta=\sum_{i\in I}\beta_ii$, $a=\max\{0,m-\beta_i\}$, $b=\min\{m,\alpha_i\}$, $P'_t=(ti,\beta-(m-t)i)$.
\end{corollary}

Derivation functors for various $\alpha\in\mathbb{N}I$ induce a derivation ${r_i}$ on $\mathbf{K}$ such that $${r_i^{m}}([A])=v^{\frac{m(m-1)}{2}}[D({\mathcal{R}_{mi}}(DA))]$$ for any $[A]\in\mathbf{K}$.

\begin{remark}
For any $i\in I$, Lusztig (Section 1.2.13 in \cite{Lusztig_Introduction_to_quantum_groups}) introduced a derivation ${r_i}:\mathbf{U}^-\rightarrow\mathbf{U}^-$ such that
\begin{equation}\label{formula_derivation_1}
{r^m_i}(xy)=\sum_{t=0}^{m}v^{(ti,\nu-(m-t)i)+t(m-t)}\frac{[m]_{v}!}{[t]_{v}![m-t]_{v}!}{r_i^t}(x){r_i^{m-t}}(y)
\end{equation}
for $y\in\mathbf{U}_{\nu}^-$.
Note that ${r_i}:\mathbf{U}^-\rightarrow\mathbf{U}^-$ is the restriction of ${r_i}:\mathbf{K}\rightarrow\mathbf{K}$ on $\mathbf{U}^-$ and
Corollary \ref{FLX_cor_1} is a categorification of Formula (\ref{formula_derivation_1}).
\end{remark}

\section{Main results}

\subsection{}

Fix $\alpha=\sum_{i\in I}\alpha_ii,\beta=\sum_{i\in I}\beta_ii,\beta'\in\mathbb{N}I$ such that $\alpha+\beta=mi+\beta'$. Fix $I$-graded vector spaces ${\mathbf{V}_{\alpha}},{\mathbf{V}_{\beta}},{\mathbf{V}_{\alpha+\beta}},{\mathbf{V}_{\beta'}}$ with dimension vectors $\alpha,\beta,\alpha+\beta,\beta'$, respectively.
Now $$\mathcal{N}=\{\lambda_t=(ti,\alpha-ti,(m-t)i,\beta-(m-t)i)|t=a,a+1,\ldots,b\},$$
where $a=\max\{0,m-\beta_i\}$ and $b=\min\{m,\alpha_i\}$.
Fix ${\mathbf{V}_{\alpha-ti}},{\mathbf{V}_{\beta-ti}}$ with dimension vectors $\alpha-ti,\beta-ti$ respectively.

For the definitions of $\mathcal{R}_{ti}$, fix $I$-graded vector subspaces $\mathbf{W}_{\alpha-ti}\subset\mathbf{V}_{\alpha},\mathbf{W}_{\beta'}\subset\mathbf{V}_{\alpha+\beta},\mathbf{W}_{\beta-ti}\subset\mathbf{V}_{\beta}$ with dimension vectors $\alpha-ti,\alpha+\beta-mi,\beta-ti$, respectively, and fix
$I$-graded linear maps $R_{\alpha-ti}'':\mathbf{V}_{\alpha-ti}\simeq \mathbf{W}_{\alpha-ti}$,
$R_{\beta'}'':\bfV_{\beta'}\simeq \mathbf{W}_{\beta'}$, and
$R_{\beta-ti}'':\bfV_{\beta-ti}\simeq \mathbf{W}_{\beta-ti}$.


\begin{theorem}\label{MT}
For any $A\in\cD_{G_{\bfV_\alpha}}^{b}(E_{\bfV_\alpha})$ and $B\in\cD_{G_{\bfV_\beta}}^{b}(E_{\bfV_\beta})$,
we have distinguished triangles
$${_{mi}\mathcal{R}}(A\ast B)_{\lambda_t}\rightarrow{_{mi}\mathcal{R}}(A\ast B)_{\geq t}\rightarrow {_{mi}\mathcal{R}}(A\ast B)_{\geq t+1}\rightarrow,$$
in $\cD_{G_{\bfV_{\beta'}}}^{b}(E_{\bfV_{\beta'}})$ for $t=a,a+1,\ldots,b-1$, such that
\begin{enumerate}
  \item[(1)]${_{mi}\mathcal{R}}(A\ast B)_{\geq a}={_{mi}\mathcal{R}}(A\ast B)$;
  \item[(2)]${_{mi}\mathcal{R}}(A\ast B)_{\lambda_t}\simeq\mathbf{1}_{pt}^{f_{m,t}(v)}\boxtimes({_{ti}\mathcal{R}}(A)\ast {_{(m-t)i}\mathcal{R}}(B))[-P_{t}](\frac{-P_{t}}{2})$ for $t=a,a+1,\ldots,b$;
  \item[(3)]${_{mi}\mathcal{R}}(A\ast B)_{\geq b}={_{mi}\mathcal{R}}(A\ast B)_{\lambda_b}$.
\end{enumerate}
\end{theorem}

\begin{proof}
Following the notations in \cite{Fang_Lan_Xiao_The_parity_of_Lusztig_restriction_functor_and_Green_formula}, consider the following diagram
$$\xymatrix{
E_{\mathbf{V}_{\alpha}}\times E_{\mathbf{V}_{\beta}} & E_{\alpha,\beta}'\ar[r]^-{p_{2\alpha,\beta}}\ar[l]_-{p_{1\alpha,\beta}} &  E_{\alpha,\beta}''\ar[r]^-{p_{3\alpha,\beta}} & E_{\mathbf{V}_{\alpha+\beta}}\\
  & & \tilde{F}\ar[u]_-{\tilde{\iota}}\ar[r]^-{\tilde{p}_3} & F_{mi,\alpha+\beta-mi}\ar[u]_-{\iota_{mi,\alpha+\beta-mi}}\ar[d]^-{\kappa_{mi,\alpha+\beta-mi}}\\
   & &  & E_{\mathbf{V}_{\beta'},}
}$$
where $\tilde{F}=E_{\alpha,\beta}''\times_{E_{\mathbf{V}_{\alpha+\beta}}}F_{mi,\alpha+\beta-mi}$
and the following diagram is a Cartesian diagram
$$\xymatrix{E_{\alpha,\beta}''\ar[r]^-{p_{3\alpha,\beta}} & E_{\mathbf{V}_{\alpha+\beta}}\\
  \tilde{F}\ar[u]_-{\tilde{\iota}}\ar[r]^-{\tilde{p}_3} & F_{mi,\alpha+\beta-mi}\ar[u]_-{\iota_{mi,\alpha+\beta-mi}}.}$$
By definition, $\tilde{F}=\{(x,\mathbf{W})\}$, where $x\in E_{\bfV_{\alpha+\beta}}$ and $\mathbf{W}\subset\bfV_{\alpha+\beta}$ is a subspace with dimension vector $\beta$ such that $\mathbf{W}$ and $\mathbf{W}_{\beta'}$ are $x$-stable (\cite{Fang_Lan_Xiao_The_parity_of_Lusztig_restriction_functor_and_Green_formula}).

Now,
\begin{eqnarray*}
{_{mi}\mathcal{R}}(A\ast B)&=&
(\kappa_{mi,\alpha+\beta-mi})_!\iota_{mi,\alpha+\beta-mi}^\ast (p_{3\alpha,\beta})_!(p_{2\alpha,\beta})_\flat p_{1\alpha,\beta}^{\ast}(A\boxtimes B)[M](\frac{M}{2})\\
&=&(\kappa_{mi,\alpha+\beta-mi})_! (\tilde{p}_{3})_!\tilde{\iota}^\ast(p_{2\alpha,\beta})_\flat p_{1\alpha,\beta}^{\ast}(A\boxtimes B)[M](\frac{M}{2}),
\end{eqnarray*}
where $M=m_{\alpha,\beta}-\langle{mi,\alpha+\beta-mi}\rangle$.

%

Let $$\tilde{F}_{\lambda_t}=\{(x,\mathbf{W})\in\tilde{F}\,\,|\,\,\underline{\dim}(\mathbf{W}\cap\mathbf{W}_{\beta'})=\beta-(m-t)i\}.$$
It is clear that $\tilde{F}$ is the disjoint of $\tilde{F}_{\lambda_t}$.
Let $\tilde{F}_{\leq t}=\bigcup_{l=a}^{t}\tilde{F}_{\lambda_l}$ and $\tilde{F}_{\geq t}=\bigcup_{l=t}^{b}\tilde{F}_{\lambda_l}$. Then $\tilde{F}_{\leq t}$ is open
and $\tilde{F}_{\geq t}$ is closed in $\tilde{F}$. At the same time,  $\tilde{F}_{\lambda_t}$ is closed in $\tilde{F}_{\leq t}$ and open in $\tilde{F}_{\geq t}$, respectively.
These embeddings are denoted by $j_{\lambda_t}:\tilde{F}_{\lambda_t}\rightarrow\tilde{F}$, $j_{\leq t}:\tilde{F}_{\leq t}\rightarrow\tilde{F}$ and $j_{\geq t}:\tilde{F}_{\geq t}\rightarrow\tilde{F}$.

For any complex $L$ on $\tilde{F}$, we have the following distinguished triangles (Section 1.4.5 in \cite{Bernstein_Lunts_Equivariant_sheaves_and_functors})
$$(j_{{\lambda_t}})_!j^\ast_{\lambda_t}L\rightarrow (j_{\geq t})_!j^\ast_{_{\geq t}}L\rightarrow (j_{\geq t+1})_!j^\ast_{_{\geq t+1}}L\rightarrow.$$

When $L=\tilde{\iota}^\ast(p_{2\alpha,\beta})_\flat p_{1\alpha,\beta}^{\ast}(A\boxtimes B)$,
we get the following distinguished triangles
\begin{equation}\label{formula_proof_1}
{_{mi}\mathcal{R}}(A\ast B)_{\lambda_t}\rightarrow{_{mi}\mathcal{R}}(A\ast B)_{\geq t}\rightarrow {_{mi}\mathcal{R}}(A\ast B)_{\geq t+1}\rightarrow,
\end{equation}
where
$${_{mi}\mathcal{R}}(A\ast B)_{\lambda_t}=(\kappa_{mi,\alpha+\beta-mi})_! (\tilde{p}_{3})_!(j_{{\lambda_t}})_!j^\ast_{\lambda_t}\tilde{\iota}^\ast(p_{2\alpha,\beta})_\flat p_{1\alpha,\beta}^{\ast}
(A\boxtimes B)[M](\frac{M}{2})$$
and
$${_{mi}\mathcal{R}}(A\ast B)_{\geq t}=(\kappa_{mi,\alpha+\beta-mi})_! (\tilde{p}_{3})_!(j_{\geq t})_!j^\ast_{\geq t}\tilde{\iota}^\ast(p_{2\alpha,\beta})_\flat p_{1\alpha,\beta}^{\ast}
(A\boxtimes B)[M](\frac{M}{2}).$$

In the proof of Theorem 3.1 in \cite{Fang_Lan_Xiao_The_parity_of_Lusztig_restriction_functor_and_Green_formula}, Fang-Lan-Xiao proved that
\begin{eqnarray}\label{formula_proof_2}
&&{_{mi}\mathcal{R}}(A\ast B)_{\lambda_t}\\
&=&(\kappa_{mi,\alpha+\beta-mi})_! (\tilde{p}_{3})_!(j_{{\lambda_t}})_!j^\ast_{\lambda_t}\tilde{\iota}^\ast(p_{2\alpha,\beta})_\flat p_{1\alpha,\beta}^{\ast}
(A\boxtimes B)[M](\frac{M}{2})\nonumber\\
&\simeq&(L_{ti}\ast L_{(m-t)i})\boxtimes({_{ti}\mathcal{R}}(A)\ast{_{(m-t)i}\mathcal{R}}(B))[-P_t](-\frac{P_t}{2})\nonumber\\
&\simeq&L^{f_{m,t}(v)}_{mi}\boxtimes({_{ti}\mathcal{R}}(A)\ast{_{(m-t)i}\mathcal{R}}(B))[-P_t](-\frac{P_t}{2})\nonumber\\
&\simeq&\mathbf{1}_{pt}^{f_{m,t}(v)}\boxtimes({_{ti}\mathcal{R}}(A)\ast{_{(m-t)i}\mathcal{R}}(B))[-P_t](-\frac{P_t}{2}).\nonumber
\end{eqnarray}

Formulas (\ref{formula_proof_1}) and (\ref{formula_proof_2}) imply the desired result.

\end{proof}

\begin{remark}
In \cite{Fang_Lan_Xiao_The_parity_of_Lusztig_restriction_functor_and_Green_formula}, Fang-Lan-Xiao constructed a collection of distinguished triangles before Lemma 3.6. In the proof of Theorem \ref{MT}, the varieties $\tilde{F}_{\lambda_t}$ have different dimensions. Hence, the distinguished triangles (\ref{formula_proof_1}) are special cases of those constructed in \cite{Fang_Lan_Xiao_The_parity_of_Lusztig_restriction_functor_and_Green_formula}.
\end{remark}

\subsection{}

Similarly, fix $\alpha,\beta,\alpha'\in\mathbb{N}I$ such that $\alpha+\beta=\alpha'+mi$. Fix $I$-graded vector spaces ${\mathbf{V}_{\alpha}},{\mathbf{V}_{\beta}},{\mathbf{V}_{\alpha+\beta}},{\mathbf{V}_{\alpha'}}$ with dimension vectors $\alpha,\beta,\alpha+\beta,\alpha'$, respectively.

\begin{theorem}
For any $A\in\cD_{G_{\bfV_\alpha}}^{b}(E_{\bfV_\alpha})$ and $B\in\cD_{G_{\bfV_\beta}}^{b}(E_{\bfV_\beta})$,
we have the following distinguished triangles in $\cD_{G_{\bfV_{\alpha'}}}^{b}(E_{\bfV_{\alpha'}})$ for $t=a+1,\ldots,b-1,b$
$${\mathcal{R}_{mi}}(A\ast B)_{\lambda_t}\rightarrow{\mathcal{R}_{mi}}(A\ast B)_{\leq t}\rightarrow {\mathcal{R}_{mi}}(A\ast B)_{\leq t-1}\rightarrow,$$
such that
\begin{enumerate}
  \item[(1)]${\mathcal{R}_{mi}}(A\ast B)_{\leq b}={\mathcal{R}_{mi}}(A\ast B)$;
  \item[(2)]${\mathcal{R}_{mi}}(A\ast B)_{\lambda_t}\simeq\mathbf{1}_{pt}^{f_{m,t}(v)}\boxtimes({\mathcal{R}_{ti}}(A)\ast {\mathcal{R}_{(m-t)i}}(B))[-P'_{t}](\frac{-P'_{t}}{2})$ for $t=a,a+1,\ldots,b$;
  \item[(3)]${\mathcal{R}_{mi}}(A\ast B)_{\leq a}={\mathcal{R}_{mi}}(A\ast B)_{\lambda_a}$.
\end{enumerate}
\end{theorem}

\subsection{}

In Section 12.2 of \cite{Lusztig_Introduction_to_quantum_groups}, Lusztig introduced a geometric pairing $\{A,B\}$ for  $A,B\in\cD_{G_{\bfV}}^{b,ss}(E_{\bfV})$, where $\bfV$ is a fixed $I$-graded vector space with dimension vector $\nu\in\mathbb{N}I$.

\begin{proposition}[Section 8.1.10 in \cite{Lusztig_Introduction_to_quantum_groups}, see also Proposition 1.14 in \cite{Schiffmann_Lectures2}]
Let $A$ and $B$ be two simple perverse sheaves in $\cD_{G_{\bfV}}^{b,ss}(E_{\bfV})$. Then
$$\{A,B\}\in\left\{\begin{array}{cc}
           1+v\mathbb{N}[[v]] & \textrm{if $A\simeq D(B)$} \\
           v\mathbb{N}[[v]] & \textrm{otherwise}.
         \end{array}
\right.$$
\end{proposition}

\begin{corollary}\label{BF_cor_2}
Let $A$ and $A'$ be two semisimple complexes in $\cD_{G_{\bfV}}^{b,ss}(E_{\bfV})$. If $\{A,B\}=\{A',B\}$  for any simple perverse sheaf $B$ in $\cD_{G_{\bfV}}^{b,ss}(E_{\bfV})$, then $A\simeq A'$.
\end{corollary}

\begin{proposition}[Lemma 12.2.2 in \cite{Lusztig_Introduction_to_quantum_groups}, see also Proposition 1.15 in \cite{Schiffmann_Lectures2}]\label{BF_prop_1}
Fix $\alpha,\beta\in\mathbb{N}I$ and $I$-graded vector spaces ${\mathbf{V}_{\alpha}},{\mathbf{V}_{\beta}},{\mathbf{V}_{\alpha+\beta}}$ with dimension vectors $\alpha,\beta,\alpha+\beta$, respectively.
For semisimple complexes $A\in\cD_{G_{\mathbf{V}_{\alpha}}}^{b,ss}(E_{\mathbf{V}_{\alpha}})$, $B\in\cD_{G_{\mathbf{V}_{\beta}}}^{b,ss}(E_{\mathbf{V}_{\beta}})$ and $C\in\cD_{G_{\mathbf{V}_{\alpha+\beta}}}^{b,ss}(E_{\mathbf{V}_{\alpha+\beta}})$, we have $$\{A\ast B,C\}=\{A\boxtimes B,\mathrm{Res}^{\alpha+\beta}_{\alpha,\beta}(C)\}.$$
\end{proposition}

\begin{corollary}\label{BF_cor_1}
Fix $\alpha,\alpha'\in\mathbb{N}I$ such that $\alpha'=\alpha+mi$ and $I$-graded vector spaces ${\mathbf{V}_{\alpha}}$ and ${\mathbf{V}_{\alpha'}}$ with dimension vectors $\alpha$ and $\alpha'$, respectively.
For $A\in\cD_{G_{\mathbf{V}_{\alpha}}}^{b,ss}(E_{\mathbf{V}_{\alpha}})$ and $B\in\cD_{G_{\mathbf{V}_{\alpha'}}}^{b,ss}(E_{\mathbf{V}_{\alpha'}})$, we have $$\{L_{mi}\ast A,B\}=\prod_{k=1}^{m}\frac{1}{1-v^{2k}}\{A,{_{mi}\mathcal{R}}(B)\}$$ and
$$\{A\ast L_{mi},B\}=\prod_{k=1}^{m}\frac{1}{1-v^{2k}}\{A,{\mathcal{R}_{mi}}(B)\}.$$
\end{corollary}

\begin{proof}
By Proposition \ref{BF_prop_1}, we have
$$\{L_{mi}\ast A,B\}=\{L_{mi}\boxtimes A,\mathrm{Res}^{\alpha'}_{mi,\alpha}B\}
=\{L_{mi}\boxtimes A,L_{mi}\boxtimes {_{mi}\mathcal{R}}(B)\}.
$$
By Section 8.1.10(f) in \cite{Lusztig_Introduction_to_quantum_groups}, we have
$$
\{L_{mi}\boxtimes A,L_{mi}\boxtimes {_{mi}\mathcal{R}}(B)\}
=\{L_{mi},L_{mi}\}\{A,{_{mi}\mathcal{R}}(B)\}.
$$
Lemma 1.4.4 in \cite{Lusztig_Introduction_to_quantum_groups} implies that
$$\{L_{mi},L_{mi}\}
=\prod_{k=1}^{m}\frac{1}{1-v^{2k}}.
$$
Hence, we get the first desired result.
The second one is similar.

\end{proof}

\begin{proposition}[Corollary 9.5 in \cite{Lusztig_Quivers_perverse_sheaves_and_the_quantized_enveloping_algebras}]\label{BF_prop_2}
For any $i\neq j$, we have
$$\bigoplus_{m+n=1-(i,j)\atop\textrm{ $m$ is odd}}L_{mi}\ast L_{j}\ast L_{ni}\simeq\bigoplus_{m+n=1-(i,j)\atop\textrm{ $m$ is even}}L_{mi}\ast L_{j}\ast L_{ni}.$$
\end{proposition}

\begin{theorem}\label{MT_2}
For any $i\neq j\in I$, we have
$$\bigoplus_{m+n=1-(i,j)\atop\textrm{ $m$ is odd}}{_{mi}\mathcal{R}}\cdot{_{j}\mathcal{R}}\cdot{_{ni}\mathcal{R}}\simeq\bigoplus_{m+n=1-(i,j)\atop\textrm{ $m$ is even}}{_{mi}\mathcal{R}}\cdot{_{j}\mathcal{R}}\cdot{_{ni}\mathcal{R}}$$ and
$$\bigoplus_{m+n=1-(i,j)\atop\textrm{ $m$ is odd}}{\mathcal{R}_{mi}}\cdot{\mathcal{R}_{j}}\cdot{\mathcal{R}_{ni}}\simeq\bigoplus_{m+n=1-(i,j)\atop\textrm{ $m$ is even}}{\mathcal{R}_{mi}}\cdot{\mathcal{R}_{j}}\cdot{\mathcal{R}_{ni}}.$$
\end{theorem}

\begin{proof}
Let $\alpha,\beta\in\mathbb{N}I$ such that $\alpha=\beta+j+(1-(i,j))i$, and ${\mathbf{V}_{\alpha}}, {\mathbf{V}_{\beta}}$ be fixed $I$-graded vector spaces with dimension vectors $\alpha, \beta$, respectively.
By Corollary \ref{BF_cor_1}, we have
$$\{\bigoplus_{m+n=1-(i,j)\atop\textrm{ $m$ is odd}}{_{mi}\mathcal{R}}\cdot{_{j}\mathcal{R}}\cdot{_{ni}\mathcal{R}}(A),B\}
=\{A,\bigoplus_{m+n=1-(i,j)\atop\textrm{ $m$ is odd}}L_{ni}\ast L_{j}\ast L_{mi}\ast B\},
$$
and
$$\{\bigoplus_{m+n=1-(i,j)\atop\textrm{ $m$ is even}}{_{mi}\mathcal{R}}\cdot{_{j}\mathcal{R}}\cdot{_{ni}\mathcal{R}}(A),B\}
=\{A,\bigoplus_{m+n=1-(i,j)\atop\textrm{ $m$ is even}}L_{ni}\ast L_{j}\ast L_{mi}\ast B\},
$$
for any $A\in\cD_{G_{\mathbf{V}_{\alpha}}}^{b,ss}(E_{\mathbf{V}_{\alpha}})$ and $B\in\cD_{G_{\mathbf{V}_{\beta}}}^{b,ss}(E_{\mathbf{V}_{\beta}})$.
By Proposition \ref{BF_prop_2},
$$\{A,\bigoplus_{m+n=1-(i,j)\atop\textrm{ $m$ is odd}}L_{ni}\ast L_{j}\ast L_{mi}\ast B\}
=\{A,\bigoplus_{m+n=1-(i,j)\atop\textrm{ $m$ is even}}L_{ni}\ast L_{j}\ast L_{mi}\ast B\}.
$$
By Corollary \ref{BF_cor_2}, we get the first desired result. The second one is similar.

\end{proof}

\subsection{}

For any $A\in\cD_{G_{\bfV_\alpha}}^{b,ss}(E_{\bfV_\alpha})$, define two maps $m^{L}_{[A]}$ and $m^{R}_{[A]}$ on the algebra $\mathbf{K}$ by $m^{L}_{[A]}([B])=[A\ast B]$ and $m^{R}_{[A]}([B])=[B\ast A]$, respectively.

For any $i\in I$, define two maps $\epsilon_{i}$ and ${_{i}\epsilon}$ on the algebra $\mathbf{K}$ by
$\epsilon_{i}([B])=\mathcal{R}_{i}([B])$ and ${_{i}\epsilon}([B])={_{i}\mathcal{R}}([B])$, respectively.

\begin{proposition}
\begin{enumerate}
  \item[(1)]For any $A\in\cD_{G_{\bfV_\alpha}}^{b,ss}(E_{\bfV_\alpha})$ and $B\in\cD_{G_{\bfV_\beta}}^{b,ss}(E_{\bfV_\beta})$, we have $$m^{L}_{[A]}\circ m^{L}_{[B]}=m^{L}_{[A\ast B]}\,\,\,\,\,\textrm{and}\,\,\,\,\,m^{R}_{[B]}\circ m^{R}_{[A]}=m^{R}_{[A\ast B]}.$$
  \item[(2)]For any $i\neq j\in I$, we have $$\sum_{m+n=1-(i,j)}(-1)^m\epsilon^{(m)}_{i}\circ\epsilon_{j}\circ\epsilon^{(n)}_{i}=0$$ and $$\sum_{m+n=1-(i,j)}(-1)^m{_{i}\epsilon}^{(m)}\circ{_{j}\epsilon}\circ{_{i}\epsilon}^{(n)}=0,$$
      where ${_{i}\epsilon}^{(n)}={_{i}\epsilon}^{n}/[n]_{v}!$ and ${\epsilon_{i}^{(n)}}={\epsilon_{i}^{n}}/[n]_{v}!$.
  \item[(3)]For any $A\in\cD_{G_{\bfV_\alpha}}^{b,ss}(E_{\bfV_\alpha})$ and $i\in I$, we have $${_{i}\epsilon}\circ m^{L}_{[A]}=v^{-(\alpha,i)}m^{L}_{[A]}\circ{_{i}\epsilon}+m^{L}_{[{_{i}\mathcal{R}}(A)]}.$$
  \item[(4)]For any $B\in\cD_{G_{\bfV_\beta}}^{b,ss}(E_{\bfV_\beta})$ and $i\in I$, we have $${\epsilon_{i}}\circ m^{R}_{[B]}=v^{-(\beta,i)}m^{R}_{[B]}\circ{\epsilon_{i}}+m^{R}_{[{\mathcal{R}_{i}}(B)]}.$$
\end{enumerate}
\end{proposition}

\begin{proof}
The first item is implied by Proposition \ref{prop_associative}. The second one is implied by Theorem \ref{MT_2}. The third and last one are implied by Corollarys \ref{FLX_cor} and \ref{FLX_cor_1}, respectively.

\end{proof}

\bibliography{mybibfile}

\end{document}